\documentclass[12pt]{amsart}
\usepackage[all]{xy}
\usepackage{mathrsfs}
\usepackage[margin=1.4in]{geometry}

\usepackage{graphicx}
\usepackage[cyr]{aeguill}\usepackage[francais,english]{babel}
\usepackage{amsmath,amsfonts,amssymb}
\usepackage{color}
\usepackage[utf8]{inputenc}
\usepackage{comment}
\usepackage[shortlabels]{enumitem}
\usepackage{etoolbox}
\usepackage{float}
\usepackage{latexsym}
\usepackage{lipsum}
\usepackage{needspace}
\usepackage{tikz}
\usepackage{hyperref}

\newtheorem{theorem}{Theorem}
\numberwithin{theorem}{section}
\numberwithin{equation}{section}
\newtheorem{lemma}[theorem]{Lemma}

\newtheorem{proposition}[theorem]{Proposition}
\newtheorem{definition}[theorem]{Definition}
\newtheorem{corollary}[theorem]{Corollary}

\newcommand{\N}{\mathbb{N}}
\newcommand{\R}{\mathbb{R}}\newcommand{\Z}{\mathbb{Z}}
\newcommand{\map}[3]{ #1 : #2 \longrightarrow #3 }

\newcommand{\mapl}[5]{ #1 : #2 \longrightarrow #3 : #4 \longmapsto #5 }

\begin{document}

\title{On semilinear sets and asymptotically approximate groups}
\author{Arindam Biswas}
\address{Universit\"at Wien, Fakult\"at f\"ur Mathematik, 
	Oskar-Morgenstern-Platz 1, 1090 Wien, Austria \& Erwin Schr\"odinger International Institute for Mathematics and Physics (E.S.I.) Boltzmanngasse 9, 1090 Wien, Austria}
\curraddr{}
\email{arindam.biswas@univie.ac.at}
\thanks{}

\author{Wolfgang Alexander Moens}
\address{Universit\"at Wien, Fakult\"at f\"ur Mathematik\\
	Oskar-Morgenstern-Platz 1, 1090 Wien, Austria. }
\email{wolfgang.moens@univie.ac.at}


\begin{abstract}
Let $G$ be any group and $A$ be an arbitrary subset of $G$ (not necessarily symmetric and not necessarily containing the identity). The $h$-fold product set of $A$ is defined as 
$$A^{h} :=\lbrace a_{1}.a_{2}...a_{h} : a_{1},\ldots,a_n \in A \rbrace.$$ 
Nathanson considered the concept of an asymptotic approximate group. Let $r,l \in \N$. The set $A$ is said to be an $(r,l)$ approximate group in $G$ if there exists a subset $X$ in $G$ such that $|X|\leqslant l$ and $A^{r}\subseteq XA$. The set $A$ is an asymptotic $(r,l)$-approximate group if the product set $A^{h}$ is an $(r,l)$-approximate group for all sufficiently large $h$. \\

Recently, Nathanson showed that every finite subset $A$ of an abelian group is an asymptotic $(r,l')$ approximate group (with the constant $l'$ explicitly depending on $r$ and $A$). We generalise the result and show that, in an arbitrary abelian group $G$, the union of $k$ (unbounded) generalised arithmetic progressions is an asymptotic $(r,(4rk)^k)$-approximate group.
	
\end{abstract}

\maketitle	
	
\section{Introduction}
Let $(G,\cdot)$ be any group and let $A_{1},A_{2},\cdots ,A_{n}\subseteq G$ be non-empty subsets of $G$. The product set $A_{1}\cdots A_{n}$ is defined as 
$$A_{1}\cdots A_{n} :=\lbrace a_{1}\cdots a_{n} : a_{1} \in A_{1}, \ldots , a_n \in A_n\rbrace.$$
In a similar way, we define
 $$A^{n} := \lbrace a_{1}\cdots a_{n} : a_{1}\in A, \ldots , a_n \in A\rbrace.$$ 
 
 \subsection{Background and motivation} In recent times, approximate groups have become a very popular object of study in mathematics. The formal definition of an approximate group was introduced by Tao in \cite{myfav9} and a part of it was motivated by its use in the work of Bourgain-Gamburd  \cite{myfav12} on super-strong approximation for Zariski dense groups of $SL_{2}(\mathbb{Z})$. 
\begin{definition}[$K$-approximate group \cite{myfav9}]
\label{AppTao}
 	Let $(G,\cdot)$ be some group and $K\geqslant 1$ be some parameter. A finite set $A\subseteq G $ is called a $K$-approximate group if 
 	\begin{enumerate}
 		\item Identity of $G$, $e\in A$.
 		\item It is symmetric, i.e. if $a\in A$ then $a^{-1}\in A$.
 		\item There is a symmetric subset $X$ lying in $A.A$ with $|X|\leqslant K$ such that $A.A\subseteq X.A$
 	\end{enumerate}
 \end{definition}

 Even though the formal definition was from 2008, the study of sets similar to the form of being an approximate group has a long history in additive number theory starting from the work of Freiman in 1964 [see Freiman's theorem - \cite{myfav4}].

Nathanson considered a more general notion of an approximate group. For him, the set $A$ need not be finite, nor symmetric, nor contain the identity. This is the notion we shall consider for the rest of the paper.
\begin{definition}[$(r,l)$-approximate group \cite{myfav118}]
\label{Appgr}
	Let $r,l \in \N$.\footnote{$\mathbb{N}$ denotes the set of positive integers, $\mathbb{N}^{*}=\mathbb{N}\cup \lbrace 0\rbrace$.} A non-empty subset $A\subseteq G$ is an $(r,l)$ approximate group if there exists a set $X\subseteq G$ such that $$|X|\leqslant l \text{ and } A^{r}\subseteq XA.$$
\end{definition}
 
We trivially note that every subset $A$ of $G$ is a $(1,1)$-approximate group. As remarked by Nathanson in \cite{myfav118}, the definition of approximate group used here is less restrictive than the one used by Tao: if $A$ is a $K$-approximate group in the sense of Definition \ref{AppTao} then it is a $(2,K)$ approximate group in the sense of Definition \ref{Appgr}. The theory of approximate groups has found numerous applications in different branches of mathematics. It is difficult to list all of them but some of the seminal works using approximate groups can be found in \cite{myfav12}, \cite{myfav22}, \cite{myfav21}, \cite{myfav31} etc.  As our motivation is somewhat different, we do not go into the details of the above results here. 

Instead, we arrive at the definition of an asymptotic approximate group, which is a subset $A\subseteq G$ such that every sufficiently large power of $A$ is an $(r,l)$ approximate group.
\begin{definition}[Asymptotic $(r,l)$-approximate group - \cite{myfav118}]\label{AAppgr}
	Let $r,l \in \N$. A subset $A$ of a group $(G,\cdot)$ is an asymptotic $(r,l)$ approximate group if there exists a threshold $h_0 \in \N$ such that for each natural number $h \geq h_0$ there exists a set $X_h$ satisfying 
	$$|X_{h}|\leqslant l \text{ and } A^{hr}\subseteq X_{h}A^{h}.$$
\end{definition}

Recently, it was shown by Nathanson that every finite subset of an abelian group is an asymptotic approximate group.
\begin{theorem}[Theorem 6, \cite{myfav118}] \label{ThmNathanson}
	Let $r,k \in \N$ and consider a finite subset $A$ of cardinality $k$ in an abelian group. Then there exists an $l \in \N$ such that $A$ is an asymptotic $(r,l)$ approximate group.
\end{theorem}

Nathanson also obtained an upper bound for $l$ of the form $n_{0}kb(r,k)$, where $n_{0}$ is the size of the torsion subgroup of the group $\langle A \rangle$ generated by $A$, and where $b(r,k)$ denotes the binomial co-efficient
$$b(r,k) = \begin{pmatrix}
	(r+1)(k-1)-1\\
	k-1
\end{pmatrix}.$$
We note that the theorem cannot be extended to free groups [see Theorem $1$ of \cite{myfav118}].

\subsection{Statement of results} In this article, we show- 

\begin{theorem}[Thm. \ref{ThmSemilinear}]\label{Thm1}
	Let $r,k \in \N$ and consider a union $A$ of $k$ unbounded generalised arithmetic progressions in an abelian group. Then $A$ is an approximate $(r,(4rk)^k)$ asymptotic group.
\end{theorem}

We note that every finite set $A$ is also (in a trivial sense) an unbounded generalised arithmetic progression. So, by specialisation, we recover Nathanson's theorem \ref{ThmNathanson}. We further note that, unlike in theorem \ref{ThmNathanson}, our constant $l = (4rk)^k$ does not depend on the torsion subgroup of $\langle A \rangle$. \\

We will deduce-

\begin{corollary}[Cor. \ref{mainresult}]\label{Thm2}
		Let $r,k \in \N$ and consider a union $A$ of $k$ (bounded or unbounded) generalised arithmetic progressions in an abelian group. Then there exists a natural number $l$ such that $A$ is an asymptotic $(r,l)$ approximate group.
\end{corollary}

To show the above, we shall first give an alternate proof of Nathanson's theorem in the case of abelian groups [see sec \ref{abAA}]. After that, we shall adapt the method of this proof in the context of the finite collection of generalised arithmetic progressions (semi-linear sets) [see sec \ref{APAA}] to show Theorem \ref{Thm1} and Corollary \ref{Thm2}. \\

\paragraph{\textbf{Notation.}} We make an important remark. In the rest of this article, we will only work with abelian groups and we will follow the convention of using the \emph{additive notation} $(G,+)$ rather than the multiplicative notation $(G,\cdot)$ for an abelian group $G$. In order to make our exposition consistent, we should therefore consider \emph{sum sets} instead of product sets. So we will use, in particular, the sum sets $$1A := A, 2A := A + A := \{ a_1 + a_2 | a_1, a_2 \in A \}, $$ and so on. This additive notation has a second advantage: it allows us to reserve from now on the notation $A^n$ for Cartesian products $$ A^n := \{ (a_1,a_2,\ldots,a_n) | a_1,\ldots,a_n \in A \}.$$

\vspace{10mm}
\section{Preliminaries}
In this section we discuss the basic notions that we shall require for the rest of the paper.

\begin{definition}[Arithmetic progressions]
	A subset $X$ of an abelian group $(G,+)$ is an \emph{unbounded} arithmetic progression if there exist $a,b \in G$ such that
	 $$X = P(a,b) := \{ a + n b | \,n \in \Z_{\geq 0} \}.$$
	 A subset $Y$ of $G$ is a \emph{bounded} arithmetic progression if there exist $a,b \in G$ and $m \in \Z_{\geq 0}$ such that $$Y = P_m(a,b) := \{ a + n b |\, n \in [0,m] \cap \Z \}.$$
\end{definition}

More generally, we will use the following objects.

\begin{definition}[Generalised arithmetic progressions]
	A subset $\mathbf{X}$ of an abelian group $(G,+)$ is an \emph{unbounded} generalised arithmetic progression of dimension $d$ if there exist $a,b_{1},b_{2},\cdots, b_{d} \in G$ such that
	$$\mathbf{X} = P(a,b_{1},\cdots, b_{d}) := \{ a + n_{1}b_{1}+\cdots + n_{d}b_{d} | \,n_{1},\cdots,n_{d} \in \Z_{\geq 0} \}.$$
	A subset $\mathbf{Y}$ of $G$ is a \emph{bounded} generalised arithmetic progression of dimension $d$ if there exist $a,b \in G$ and $m \in \Z_{\geq 0}$ such that 
	$$\mathbf{Y} =  P_{m_{1},\cdots, m_{d}}(a,b_{1},\cdots, b_{d}) := \{ a + n_{1}b_{1}+\cdots + n_{d}b_{d} | \, n_{i} \in [0,m_{i}] \cap \Z, 1\leqslant i\leqslant d \}.$$	
\end{definition}

A generalised arithmetic progression is also known as a linear set. Thus we have the notion of bounded linear set and unbounded linear set.

\begin{definition}[Semilinear set]
	A finite union of unbounded linear sets is called a semilinear set.
\end{definition}

\begin{definition}[External covering number]
	Consider an abelian group $(G,+)$ and three subsets $S_1,S_2,T$. If $S_1$ can be covered by finitely-many $T$-translates of $S_2$, then we define the covering number $$\kappa_{G,T}(S_1,S_2) := \min \lbrace  k \in \N \mid \exists t_1 , \ldots , t_k \in T : S_1 \subseteq \bigcup_{1 \leq i \leq k} (t_i + S_2) \rbrace.$$
\end{definition}

\begin{definition}[Simplex; cube; dilation]
	Let $k  \in \N \cup \{ 0 \}$ and $\rho \in \R_{\geq 0}$. We define $$\Delta_k(\rho) := \{ (v_1,\ldots,v_k) \in \R_{\geq 0}^k \mid v_1 + \cdots + v_k \leq \rho \} \subseteq \R^k.$$ We also define the standard (closed) cube  $$C_k(\rho) := \{ (v_1,\ldots,v_k) \in \R_{\geq 0}^k | v_1,\ldots,v_k \leq \rho\}$$ and the open cube 
	$$C_k(\rho)^\circ := \{ (v_1,\ldots,v_k) \in \R_{\geq 0}^k | v_1,\ldots,v_k < \rho\}.$$
	For a subset $Y$ of $\R^k$, we define the \emph{dilation} $\rho \ast Y := \{ r \cdot v \mid v \in Y \}.$
\end{definition}

We trivially note that $\rho \ast \Delta_k(1) = \Delta_k(\rho)$.

\vspace{10mm}
\section{Every finite subset in an abelian group is an asymptotic approximate group}\label{abAA}

In this section, we give an alternate proof of Nathanson's result.
\begin{lemma}[Covering the simplex] \label{EffTopCov}
	Let $k,r \in \N $. Then $\Delta_{k}(r)$ can be covered by $(2rk)^k$-many $\R^{k}$-translations of the open cube $C_{k}(1/k)^\circ$ in $\R^{k}$.  
\end{lemma}

\begin{proof}
	Since $C_k(1/2k)$ can be translated into $C_k(1/k)^\circ$, and since $\Delta_k(r)$ is contained in $C_k(r)$, we need only cover $C_k(r)$ with $(2rk)^k$-many translations of $C_k(1/2k)$. And doing this is quite straight-forward. 
\end{proof}

We see, in particular, that  $\kappa_{\R^{k},\R^{k}}(\Delta_{k}(r),C_{k}(1/k)^\circ)$ is well-defined. 

\begin{lemma}[Covering in the lattice] \label{LemCovLat}
	Let $k,r,h \in \N$ with $h > 2 k$. Then we can cover $\Delta_{k}(r h) \cap \Z^{k}$ with finitely-many $\Z^{k}$-translates of $\Delta_{k}(h)^\circ \cap \Z^{k}$. In fact, we have the bound \begin{equation} \kappa_{\Z^{k},\Z^{k}}(\Delta_{k}(r h) \cap \Z^{k},\Delta_{k}(h)^\circ \cap \Z^{k}) \leq (4rk)^k . \label{EqCoverBd} \end{equation}
\end{lemma}

\begin{proof}
	Let us use the abbreviation $\kappa := (2rk)^k$. According to lemma \ref{EffTopCov}, we can select $v_1,\ldots,v_\kappa \in \R^{k}$ such that $$\Delta_{k}(r) \subseteq \bigcup_{1 \leq i \leq \kappa} (v_i + C_{k}(1/k)^\circ).$$ By using the dilation $x \mapsto h \ast x$ in $\R^{k}$, we obtain 
	\begin{eqnarray}
	\Delta_{k}(rh) \subseteq \bigcup_{ 1 \leq i \leq \kappa } ( h \cdot v_i + C_{k}(h/k)^\circ ). \label{TranslScale}
	\end{eqnarray} 
	We note that the vectors $h \cdot v_i = (h v_{i,1},\ldots,h v_{i,k-1})$ need not have \emph{integer} coordinates. But we have assumed that $h/k > 2$. So, for each $h \cdot v_i$, there exist (not necessarily distinct) elements $w_{i,1} , \ldots , w_{i,2^k} \in \Z^k$ such that \begin{equation} h \cdot v_i + C_k(h/k)^\circ \subseteq \bigcup_{1 \leq j \leq 2^k} (w_{i,j} + C_k(h/k)^\circ) \subseteq \bigcup_{1 \leq j \leq 2^k}(w_{i,j} + \Delta_k(h)^\circ) \label{Wobble}.\end{equation} 
	So, by combining \ref{TranslScale} and \ref{Wobble}, we obtain	$$
	\Delta_k(rh) \subseteq \bigcup_{1 \leq i \leq \kappa} \bigcup_{1 \leq j \leq 2^k} (w_{i,j} + \Delta_k(h)^\circ)
	.$$	By intersecting with $\Z^k$, we finally obtain \ref{EqCoverBd}: $$\Delta_k(rh) \cap \Z^k \subseteq 
	\bigcup_{1 \leq i \leq \kappa} \bigcup_{1 \leq j \leq 2^k} (w_{i,j} + (\Delta_k(h)^\circ \cap \Z^k)) .$$
\end{proof}

\begin{proposition}[Free abelian groups] \label{ExAbelian}
	Let $k,r \in \N$. Consider the free abelian group $(\Z^k,+)$ with a free generating set ${B} := \{e_1,\ldots,e_k\} $. Then, for every natural number $h > 2k$, we have $$ \kappa_{\Z^k,\Z^k}(rh {B} , {h B} \cap \Z_{> 0}^k) \leq (4rk)^k.$$
\end{proposition}

\begin{proof}	
	We may suppose that $k > 1$, since otherwise we trivially have $rhB = \{ rh \}$ and $hB \cap \Z_{> 0} = \{ h \}$, so that indeed $$rhB = (rh-h) + (hB \cap \Z_{>0}).$$ With this assumption, the $ e_1,\ldots,e_{k-1}$ generate the free abelian subgroup $\Z^{k-1}$ of $\Z^{k}$. Let us abbreviate $$\kappa :=  (4r(k-1))^{(k-1)} \leq (4rk)^k.$$ According to lemma \ref{LemCovLat}, we can find $v_1, \ldots, v_\kappa \in \Z^{k-1}$ such that \begin{equation}
	\Delta_{k-1}(r h) \cap \Z^{k-1} \subseteq \bigcup_{1 \leq i \leq \kappa} \left( v_i + (\Delta_{k-1}(h)^\circ \cap \Z^{k-1}) \right). \label{Zk-1Cov}
	\end{equation} We now apply the affine map $$\mapl{f}{\Z^{k-1}}{\Z^k}{(y_1,\ldots,y_{k-1})}{(y_1,\ldots,y_{k-1} , h r - (y_1 + \cdots + y_{k-1}) )}$$ to \ref{Zk-1Cov} and obtain
	\begin{eqnarray*}
		rh B &=& f(\Delta_{k-1}(rh) \cap \Z^{k-1}) \\
		&\subseteq& \bigcup_{1 \leq i \leq \kappa} f \left( v_i + (\Delta_{k-1}(h)^\circ \cap \Z^{k-1} )\right) \\
		&=& \bigcup_{1 \leq i \leq \kappa} \left( (f(v_i) - (0,\ldots,0,h) ) + ({h B} \cap \Z_{> 0}^k) \right).
	\end{eqnarray*}
	This finishes the proof. 
\end{proof}

\begin{lemma}[Free abelian groups ctd.] \label{ExAbelianBis}
	Let $k,r \in \N$. Consider the free abelian group $(\Z^k,+)$ with a free generating set ${B} $. Then, for every natural number $h$, we have $$ \kappa_{\Z^k,\Z^k}(rh {B} , h {B}) \leq (4rk)^k.$$
\end{lemma}

\begin{proof}	
	We may assume that $k > 1$, since otherwise we trivially have $$\kappa_{\Z^k,\Z^k}(rh {B} , h {B}) = 1 \leq (4rk)^k.$$ If $h \leq 2k$, then $$\kappa_{\Z^k,\Z^k}(rh {B} , h {B}) \leq |rhB| \leq (2rk+1)^k \leq (4rk)^k.$$ Otherwise, $h > 2k$, so that we may use Prop. \ref{ExAbelian} to also conclude that $$\kappa_{\Z^k,\Z^k}(rhB,hB) \leq \kappa_{\Z^k,\Z^k}(rhB,hB \cap \Z_{>0}^k) \leq (4rk)^k.$$ 
\end{proof}

\begin{proposition}[Abelian case]
	Consider an abelian group $(G,+)$ with a subset $A$ of finite cardinality $k > 0$. Then, for all natural numbers $r$ and $h$, there exists a subset $X$ of $G$ such that $$\vert X \vert \leq (4 r k)^{k} $$ and such that $$ r \cdot (h \cdot A) \subseteq X + (h \cdot A).$$
\end{proposition}

\begin{proof}
	Let $B$ be a free generating set of the free abelian group $(\Z^k,+)$. According to lemma \ref{ExAbelianBis}, we can find a subset $Y$ of $\Z^k$ of size at most $(4 r k)^{k}$ such that 
	$$rh B \subseteq Y + h B.$$ The fundamental theorem of $k$-generated abelian groups now gives us a group homomorphism $\map{\pi}{\Z^k}{G}$ such that $\pi(B) = A$. Define $X := \pi({Y})$. Then $$\vert X \vert \leq \vert {Y} \vert \leq (4 r k)^{k}.$$ By projection, we also obtain $$ r \cdot (h \cdot A) = \pi(r \cdot (h \cdot {B})) \subseteq \pi({Y} + h \cdot {B}) = X + (h \cdot A).$$
\end{proof}

\begin{corollary}
\label{Natresult}
	Let $r,k \in \N$ and let $A$ be a finite subset of cardinality $k$ in an abelian group. Then $A$ is an asymptotic $(r,(4 r k)^{k})$-approximate group.
\end{corollary}



\vspace{10mm}
\section{Finite union of arithmetic progressions}\label{APAA}


\begin{proposition}[Semi-linear case in free abelian groups] \label{SemiLinear}
	Let $k,l_1,\ldots,l_k \in \N$ and set $l := l_1 + \cdots + l_k$. Consider the free abelian group $(\Z^{k+l},+)$ of rank $k+l$ with standard basis $$f_1,\ldots,f_k,e_{1,1},\ldots,e_{1,l_1},e_{2,1},\ldots,e_{k,l_k}.$$ Then the union $$A := P(f_1,e_{1,1},\ldots,e_{1,l_1}) \cup \cdots \cup P(f_k,e_{k,1},\ldots,e_{k,l_k})$$ is an asymptotic approximate group. More precisely: for every natural number $r$ and every $h > 2 k$, we have the bound $$\kappa_{\Z^{2k},\Z^{2k}}(rhA,hA) \leq (4rk)^k.$$
\end{proposition}

\begin{proof}
	Let $r , h \in \N$ be arbitrary natural numbers with $h > 2k$. Let us introduce some notation. We define the finite set $F := \{ f_1,\ldots,f_k \}$ and, for each $i \in \{ 1,\ldots,k \}$, we define the ``quadrant'' $$E_k := \{ s_1 e_{i,1} + \cdots + s_{l_i} e_{i,l_i} | s_1 ,\ldots,s_{i,l_i} \in \Z_{\geq 0} \}.$$ Then $P(f_i,e_{i,1},\ldots,e_{i,l_i}) = f_i + E_i$. 
	We also set $$E := E_1 + \cdots + E_k  = \Z^l \subseteq \Z^{k+l}.$$ For each $\overrightarrow{n} := (n_1,\ldots,n_k) \in \Z^k$, we also define the set $T_{\overrightarrow{n}}$ as follows. Let $n_{i_1},\ldots,n_{i_t}$ be the non-zero coordinates of $\overrightarrow{n}$. Then $$T_{\overrightarrow{n}} := E_{n_{i_1}} + \cdots + E_{n_{i_t}} \subseteq E \subseteq \Z^{k+l}.$$  
	We now note that 
	\begin{eqnarray}
	hA &=& \bigcup_{\substack{\overrightarrow{n} \in \Z_{\geq 0}^k \\n_1 + \cdots + n_h = h}} (n_1 f_1 + \cdots + n_k f_k + T_{\overrightarrow{n}}) 
	\supseteq \overline{hF} + E,\label{hASupset} \end{eqnarray} 
	where $$\overline{hF} := \bigcup_{\substack{\overrightarrow{n} \in \Z_{\geq 0}^k \\n_1 + \cdots + n_h = h\\n_1 \cdots n_k \neq 0}} (n_1 f_1 + \cdots + n_k f_k).$$ Similarly, we have  
	\begin{eqnarray}
	rhA &=& \bigcup_{\substack{\overrightarrow{n} \in \Z_{\geq 0}^{k} \\n_1 + \cdots + n_{rh} = rh}} (n_1 f_1 + \cdots + n_k f_k + T_{\overrightarrow{n}}) \nonumber\\
	&\subseteq& rh F + E. \label{rhAIncl}
	\end{eqnarray}
	Set $\kappa = (4rk)^k$. According to example \ref{ExAbelian}, we can find elements $v_1 ,\ldots,v_\kappa \in \Z^k \oplus 0 \subseteq \Z^{k} \oplus \Z^l$  such that \begin{equation}
	rhF \subseteq \bigcup_{1 \leq i \leq \kappa}(v_i + \overline{hF}). \label{CoveringrhF}
	\end{equation} By combining \ref{hASupset}, \ref{rhAIncl} and \ref{CoveringrhF}, we obtain $$rhA \subseteq rhF + E \subseteq \bigcup_{1 \leq i \leq \kappa}(v_i + \overline{hF}) + E \subseteq \bigcup_{1 \leq i \leq \kappa}(v_i + hA).$$
\end{proof}


\begin{proposition} \label{PropUnionAP}
	Let $r , k \in \N$. In an abelian group, the union of $k$ unbounded arithmetic progressions is an asymptotic $(r,(4rk)^k)$-approximate group.
\end{proposition}


In fact, we will prove a slightly more general statement.

\begin{theorem} \label{ThmSemilinear}
	Let $r , k \in \N$. In an arbitrary abelian group, the union of $k$ unbounded linear sets is an asymptotic $(r,(4rk)^k)$-approximate group.
\end{theorem}

\begin{proof}
	Let $$P(F_1,E_{1,1},\ldots,E_{1,l_1}), \ldots, P(F_k,E_{k,1},\ldots,E_{k,l_k})$$ be the unbounded linear sets in the abelian group $(G,+)$ and let $B$ be their union. Set $l := l_1 + \cdots + l_k$. The fundamental theorem of $(k+l)$-generated, abelian groups (with standard basis $f_1,\ldots,f_k,e_{1,1},\ldots,e_{k,l_k}$) then gives us a homomorphism $$\map{\pi}{\Z^{k+l}}{G}$$ such that
	$$\pi(e_{1,1}) = E_{1,1} , \ldots , \pi(e_{k,l_k}) = E_{k,l_k} \text{ and }\pi(f_1) = F_1 , \ldots , \pi(f_k) = F_k.$$	Define $$A := P(f_1,e_{1,1},\ldots,e_{1,l_1}) \cup \cdots \cup P(f_k,e_{k,1},\ldots,e_{k,l_k}).$$ Then $\pi$ maps each $P(f_i,e_{i,1},\ldots,e_{i,l_l})$ onto  $P(F_i,E_{i,1},\ldots,E_{i,l_l})$, so that $\pi(A) = B$. We next define the threshold $h_0 := 2 k+1$ and we let $h$ be a natural number, at least $h_0$. According to Proposition \ref{SemiLinear}, there exist elements $w_1 , \ldots , w_{(4rk)^k} \in \Z^{k+l}$ such that 
	$$rhA \subseteq \bigcup_{1 \leq j \leq (4rk)^k} (w_j + hA).$$ By applying the homomorphism, we obtain $$ rhB = \pi(rhA) \subseteq \pi\left(\bigcup_{1 \leq j \leq (4rk)^k} (w_j + hA) \right) = \bigcup_{1 \leq j \leq (4rk)^k} (\pi(w_j) + hB) .$$
\end{proof}

By specialising the constants $l_1 = \cdots = l_k = 1$, we recover Proposition \ref{PropUnionAP}.


\begin{corollary}
\label{mainresult}
	 Let $r,k \in \N$ and consider a union $A$ of $k$ (bounded or unbounded) generalised arithmetic progressions in an abelian group. Then there exists a natural number $l$ such that $A$ is an asymptotic $(r,l)$ approximate group.
\end{corollary}

Informally: semi-linear subsets of abelian groups are asymptotic approximate groups.

\begin{proof}
	Any \emph{bounded} arithmetic progression $P_m(a,b)$ is the union of the $m+1$ \emph{unbounded} arithmetic progressions $P(a,0),P(a+b,0),\ldots$ and $P(a+mb,0)$. More generally, every \emph{bounded} linear set is the union of finitely-many \emph{unbounded} linear sets. So using Theorem \ref{ThmSemilinear} we obtain the claim.
\end{proof}

\vspace{10mm}
\section{Concluding remarks}




We conclude with a simple remark: a group $G$ is virtually-nilpotent if it is generated by a finite, asymptotic $(r,l)$ approximate group (with $r \geq 2$). Indeed: lemma \ref{LemmaGrowth} shows that the growth of $G$ is polynomial, so that we may use Gromov's characterisation of virtually-nilpotent groups [see \cite{myfav120}].

\begin{lemma}\label{Lem5.2}
	Let $f : \mathbb{N}\rightarrow \mathbb{N}$ be a monotonically increasing function. Suppose that there exist natural numbers $r,L$ at least $2$ such that for all $h \in \N$ we have the property $$\frac{f(rh)}{f(h)} < L.$$
	Then $f$ is bounded from above by some polynomial function. 
\end{lemma}
\begin{proof}
	We have, 
	\begin{align*}
		f(h) & = f(r^{\log_{r}h})\\
			 & \leqslant f(r^{\lceil \log_{r}h\rceil}) \text{ (f is monotonic)}\\
			 & = f(r.r^{\lceil \log_{r}h\rceil - 1})\\
			 & \leqslant L f(r^{\lceil \log_{r}h\rceil - 1})  \text{ (by the hypothesis on f)}\\
			 & \leqslant L^{\lceil \log_{r}h\rceil}f(1) \text{ (by induction)}\\
			 & \leqslant L.L^{\frac{\log_{L}h}{\log_{L}r}}f(1)\\
			 & \leqslant Lf(1)h^{\frac{1}{\log_{L}r}}.
	\end{align*}
	Thus $f(h)$ is bounded above by some polynomial in $h$ with maximum degree $\lceil\frac{1}{ \log_{L}r }\rceil $.
\end{proof}

\begin{lemma} \label{LemmaGrowth}
	Let $r,L \in \N$ with $r \geq 2$. Let $G$ be an arbitrary group and $A$ be a finite subset of $G$. Suppose the growth rate of the iterated powers, $A,A^{2},\cdots, A^{n},\cdots$ is super-polynomial. Then $A$ cannot be an asymptotic $(r,L)$ approximate group.
\end{lemma}
\begin{proof}
	 For each $n \in \N$, we define $f(n) := |A^n|$. We suppose that $A$ is an asymptotic $(r,L)$ approximate group and we will show that $f(n)$ is bounded by a polynomial in $n$. By assumption, there exists a threshold $h_0 \in \N$ such that for every $h \in \N$ with $h \geq h_0$ there exists a subset $X_h \subset G$ satisfying $|X_{h}|\leqslant L$ and 
	\begin{align*}
			A^{rh} & \subseteq X_{h}.A^{h}\\
			 \Rightarrow |A^{rh}| & \leqslant L.|A^{h}|\\
			\Rightarrow f(rh) & \leqslant Lf(h).			
	\end{align*}
	We wish to apply the strategy of Lemma \ref{Lem5.2}. However, the above is only known to hold for all $h\geqslant h_{0}$. So we cannot apply it directly. Let 
	$$M:= \max\lbrace f(h) : 1\leqslant h\leqslant r^{\lceil \log_{r}h_{0}\rceil}\rbrace.$$
	Consider $f(h), h\in\mathbb{N}$. Let  $k =  \max \lbrace \lceil \log_{r}h\rceil  - \lceil\log_{r}h_{0}\rceil , 0 \rbrace $. We have,
	\begin{align*}
		f(h) & = f(r^{\log_{r}h})\\
			& \leqslant f(r^{\lceil \log_{r}h\rceil}) \\
			& \leqslant L^{k} f(r^{\lceil \log_{r}h\rceil - k}) \\
			& \leqslant L^{k}M\\
			& \leqslant Mh^{\frac{1}{\log_{L}r}}.
	\end{align*}
	We see that $f(h)$ is bounded above by a polynomial of degree at most $\lceil \frac{1}{\log_{L}r}\rceil$ in $h$. This contradiction finishes the proof. 
	\end{proof}

\section{Acknowledgements}
We are grateful to Christopher Cashen for a number of helpful discussions and especially for pointing out Lemma \ref{Lem5.2}. The first author would like to acknowledge the fellowship of the Erwin Schr\"odinger International Institute for Mathematics and Physics (ESI) and would also like to thank the Fakult\"at f\"ur Mathematik, Universit\"at Wien where a part of the work was carried out. The second author would like to acknowledge the FWF grant P30842-N35 of the Austrian Science Fund.

\providecommand{\bysame}{\leavevmode\hbox to3em{\hrulefill}\thinspace}
\providecommand{\MR}{\relax\ifhmode\unskip\space\fi MR }
\providecommand{\MRhref}[2]{%
	\href{http://www.ams.org/mathscinet-getitem?mr=#1}{#2}
}
\providecommand{\href}[2]{#2}

\end{document}